\documentclass{amsart}[12pt]
\usepackage{mathrsfs, amsmath,amssymb}
\usepackage{fullpage}
\usepackage{graphicx}

\newtheorem{pro}{Proposition}
\newtheorem{lem}{Lemma}
\newtheorem{cor}{Corollary}
\newtheorem*{rem}{Remark}

\title{Inverse of the flow and moments of the free Jacobi process associated with a single projection}

\author[N. Demni]{Nizar Demni}
\address{IRMAR, Universit\'e de Rennes 1\\ Campus de
Beaulieu\\ 35042 Rennes cedex\\ France}
\email{nizar.demni@univ-rennes1.fr}

\author[T. Hamdi]{Tarek Hamdi}
\address{Department of Management Information Systems, College of Business Administration \\ Al-Rass, Al-Qassim University \\ Kingdom Of Saudi Arabia 
and Laboratoire d'Analyse Mathématiques et applications, \\ LR 11ES11 \\ Universit\'e de Tunis El-Manar, Tunisie
}
\email{hamditarek@hotmail.fr}

\begin{document}
\maketitle
\begin{abstract}
This paper is a companion to a series of papers devoted to the study of the spectral distribution of the free Jacobi process associated with a single projection. Actually, we notice that the flow derived in \cite{Dem-Hmi} solves a radial L\"owner equation and as such, the general theory of L\"owner equations implies that it is univalent in some connected region in the open unit disc. We also prove that its inverse defines the Aleksandrov-Clark measure at $z=1$ of some Herglotz function which is absolutely-continuous with an essentially bounded density. As a by-product, we deduce that $z=1$ belongs only to the discrete spectrum of the unitary operator whose spectral dynamics are governed by the flow. Moreover, we use a previous result due to the first author in order to derive an explicit, yet complicated, expression of the moments of both the unitary and the free Jacobi processes. The paper is closed with some remarks on the boundary behavior of the flow's inverse.  

\end{abstract}

\section{Introduction}
Motivated by random matrix theory and especially by the large-size limit of the matrix-valued Jacobi process (\cite{Dou}), the first author introduced the free Jacobi process $(J_t)_{t \geq 0}$ and performed its free stochastic analysis (\cite{Dem}). This is a family of non negative and bounded operators defined as the radial part of the compression of the free unitary Brownian motion by two free (in Voiculescu's sense) orthogonal projections in a finite von Neumann algebra. Up to a normalization, the free Jacobi process can be viewed as Voiculescu's liberation process associated with two projections since the free unitary Brownian motion $(Y_t)_{t \geq 0}$ converges weakly as $t \rightarrow \infty$ to a Haar-distributed unitary operator (\cite{Voi}). When both projections coincide, the series of papers \cite{DHH}, \cite{Dem-Hmi} and \cite{Demni} aim to determine the Lebesgue decomposition of the spectral distribution of $J_t$ at any fixed time $t > 0$. In particular, if the continuous rank of the underlying projection is $1/2$, then it is proved in \cite{DHH} that this probability measure is (up to an affine transformation) nothing else but the spectral distribution of the real part of $Y_{2t}$. 

For arbitrary continuous ranks, the problem of describing the spectral distribution of $J_t$ becomes considerably more difficult. In order to tackle it, the moments of $J_t$ were related to those of a unitary operator $U_t$ built out of $Y_t$ and of the symmetry associated with the orthogonal projection (\cite{Dem-Hmi}). Consequently, the rank of this symmetry vanishes when the rank of the projection equals $1/2$ and in this case, $U_t$ and $Y_{2t}$ shares the same spectrum (\cite{DHH}). Note that this last result was subsequently generalized in \cite{Izu-Ued} to the case of two projections with equal rank $1/2$ and to arbitrary distributions of $U_0$. Note also that at the analytic level, the moment generating functions of $U_t$ and of $J_t$ (the latter is viewed as an element of the compressed algebra) are related to each other by the Szeg\"o map. Coming back to the single projection case with arbitrary rank, it is proved in \cite{Dem-Hmi} that the discrete spectrum of $U_t, t > 0,$ contains $z=1$ with a weight given by the absolute value of the rank of the symmetry. It is also proved in the same paper that $z=-1$ does not belong to the discrete spectrum of $U_t$ and both results entirely determine the discrete spectrum of $J_t$. 

The approach undertaken in \cite{Dem-Hmi} and leading to these spectral data is mostly analytic. Actually, apart from the free stochastic calculus machinery developed in \cite{BenL} and used to derive a partial differential equation (hereafter pde) for the spectral dynamics of $(U_t)_{t \geq 0}$, the method of characteristics allowed to determine a flow so far defined, at any fixed time $t > 0$, in a sub-interval of $(-1,1)$ depending on $t$. However, the description of the continuous spectrum of $J_t$ necessitates the investigation of the analytic extension of the flow in the open unit disc. Indeed, it was shown in \cite{Dem-Hmi} using the Stieltjes inversion formula that the density of the spectral distribution of $J_t$ is entirely determined by the boundary behavior of the Herglotz transform of the spectral distribution of $U_t$ (in the unit disc). Tempting to prove the analytic extension of the inverse of the flow, the first author computed its Taylor coefficients around the origin (up to elementary conformal mappings) using Lagrange inversion formula (\cite{Demni}). Unfortunately, the obtained expression is quite involved and contains nested alternating sums, which prevents from obtaining the precise value of the convergence radius of the Taylor series. 

In this paper, we realize that the sought properties of the flow (analytic extension, univalence, range) may be derived from general facts on L\"owner equations. Indeed, we notice that one of the two coupled ordinary differential equations solved by the flow in \cite{Dem-Hmi} is the radial L\"owner equation driven by the Herglotz transform of the spectral distribution of $U_t$. Hence, for any $t > 0$, the flow is the Riemann map of the connected component of its analyticity region containing the origin (\cite{Law}). Using the  characteristics equation obtained in \cite{Dem-Hmi}, we also relate the inverse of flow to the Aleksandrov-Clark measure at $z=1$ of some Herglotz function which is absolutely-continuous (with respect to the Lebesgue measure on the unit circle) with an essentially bounded density. Actually, the existence of such a density is ensured by the fact we prove below that the domain where the flow is univalent is at a positive distance from $z=1$. On the other hand, we can view this relation as giving an analytic extension of the inverse of the flow in the open unit disc, then use it to prove directly (without appealing to L\"owner equations) that this extension is univalent there. Afterwards, we prove that $z=1$ belongs only to the discrete spectrum of $U_t, t > 0$, which was already noticed in \cite{Dem-Hmi} in the large-time limit $t = \infty$ using the explicit Lebesgue decomposition of the spectral distribution of $U_{\infty}$. Finally, we derive an explicit, yet complicated, expression of the moments of $U_t$ and of $J_t$. 

The paper is organized as follows. For sake of self-containedness, we settle in the next section some notations needed in the remainder of the paper and recall the main results proved in \cite{DHH}, \cite{Dem-Hmi} and \cite{Demni}. In the third section, we prove that the flow is univalent in some connected region of the open unit disc containing $z=0$ and show that its inverse defines an Aleksandrov-Clark measure at $z=1$. We also prove in the same section that the Herglotz transform associated with this measure is absolutely-continuous with respect to Lebesgue measure on the unit circle and that its density is essentially bounded. The fourth section is devoted to the proof of the fact that that $z=1$ belongs only to the discrete spectrum of $U_t, t > 0,$ and to the derivation of the moments of both $U_t$ and $J_t$. Finally, we close the paper with some remarks on the boundary behavior of the inverse of the flow. More precisely, we determine the intersection of the boundary of its range with the real axis and discuss the boundary equation.

\section{Reminder and notations}
Let $\mathscr{A}$ be a unital von Neumann algebra endowed with a finite trace $\tau$ and an involution $\star$. Let $P \in \mathscr{A}$ be an orthogonal projection with continuous rank $\tau(P)$ and assume without loss of generality that there exists a free unitary Brownian motion $Y = (Y_t)_{t \geq 0} \in \mathscr{A}$ which is free (in Voiculescu's sense) with $P$ (\cite{Biane}). Then, the free Jacobi process associated with the projection $P$ is defined by (\cite{Dem}):
\begin{equation*}
J_t := PY_tPY_t^{\star}P, \quad t \geq 0, 
\end{equation*}
with values in the compressed von Neumann algebra: 
\begin{equation*}
\left(P\mathscr{A}P, \frac{1}{\tau(P)} \tau\right). 
\end{equation*}
When $\tau(P) = 1/2$, the spectral distribution of $J_t$ at any fixed time $t > 0$ coincides with that of 
\begin{equation*}
\frac{Y_{2t} + Y_{2t}^{\star} + 2{\bf 1}}{4}
\end{equation*}
in the algebra $(\mathscr{A}, \tau)$, where ${\bf 1}$ is the unit of $\mathscr{A}$. In \cite{DHH}, two proofs of this description were given, one of them relies on the following expansion: let $S = 2P-{\bf 1}$ be the orthogonal symmetry associated with $P$ and set $\kappa := \tau(S) = 2\tau(P) - {\bf 1}$. Then, one has for any $n \geq 1$:
\begin{align}\label{Binom}
\tau[(PY_tPY_t^{\star})^n] = \frac{1}{2^{2n}}\tau[(({\bf 1} + S)Y_t({\bf 1} + S)Y_t^{\star})^n]  = \frac{1}{2^{2n+1}}\binom{2n}{n}  + \frac{\kappa}{2} + \frac{1}{2^{2n}}\sum_{k=1}^n \binom{2n}{n-k}\tau((U_t)^k),
\end{align}
where $U_t := SY_tSY_t^{\star}$. Since $S^2 = {\bf 1}$ and $Y$ is unitary, then $U$ is unitary as well and if further $\tau(S) = 0$, then $U_t$ and $Y_{2t}$ shares the same spectral distribution (\cite{DHH}). More generally, let
\begin{equation}\label{Herg}
H_{\kappa, t}(z) := \int_{\mathbb{T}} \frac{w+z}{w-z}\nu_{\kappa, t}(dw) = 1+2\sum_{n \geq 1}\tau(U_t^n)z^n
\end{equation}
be the Herglotz transform of the spectral distribution $\nu_{\kappa,t}$ of $U_t$, where $\mathbb{T}$ is the unit circle. Then, $(t,z) \mapsto H_{\kappa,t}(z)$ is the unique analytic solution in $\mathbb{D}$ of the pde (\cite{Dem-Hmi}, Proposition 3.3):
\begin{equation}\label{Pde}
\partial_t H_{\kappa, t} = -\frac{z}{2}\partial_z \left[H_{\kappa,t}^2 - \kappa^2 H_0^2\right], 
\end{equation}
where the initial value
\begin{equation*}
H_{\kappa,0}(z) := H_0(z) := \frac{1+z}{1-z}
\end{equation*}
is the Herglotz transform of $\mu_{\kappa,0} = \delta_1$. Moreover, the key result proved in \cite{Dem-Hmi} is the following characteristics equation: for any $t > 0$ and any $\kappa \in (-1,1)$, there exists a real $z_{\kappa,t} \in (0,1)$ and a flow $z \mapsto \psi_{\kappa,t}(z)$ defined in $(-1,z_{\kappa,t})$ such that
\begin{align}\label{Dyn}
[H_{\kappa,0}(z)]^2  - [H_{\kappa,\infty}(z)]^2 = [H_{\kappa, t}(\psi_{\kappa, t}(z))]^2 - [H_{\kappa,\infty}(\psi_{\kappa,t}(z))]^2.
\end{align}   
Here,
\begin{equation*}
H_{\kappa,\infty}(z) = \sqrt{1+ \frac{4\kappa^2z}{(1-z)^2}}
\end{equation*}
is the Herglotz transform  of the weak limit $\nu_{\kappa,\infty}$ of $\nu_{\kappa, t}$\footnote{There is a misprint in the statement of Corollary 2.2. in \cite{Dem-Hmi}.}:
\begin{equation*}
\nu_{\kappa,\infty} = |\kappa|\delta_1 + \sqrt{1 - \frac{\kappa^2}{\sin^2(\theta/2)}} {\bf I}_{\{|\sin(\theta/2)| \geq |\kappa|\}} d\theta
\end{equation*}
In particular, the support of the density of $\nu_{\kappa,\infty}$ is disconnected at $z=1$ as soon as $\kappa \neq 0$. Note in passing that this density shows up (up to a constant) in the large-size limit of the spectral distribution of the Hua-Pickrell model with a real parameter (see \cite{BNR}, p.4383, see also \cite{Sim}, p.87). In order to recall the expression of the flow $\psi_{\kappa,t}$, we introduce the following maps: let
\begin{equation*}
\alpha: z \mapsto \frac{1-\sqrt{1-z}}{1+\sqrt{1-z}} =  \frac{z}{(1+\sqrt{1-z})^2},
\end{equation*}
where the principal branch of the square root is taken. This is an analytic one-to-one map from the cut plane $z \in \mathbb{C} \setminus [1,\infty[$ onto the open unit disc $\mathbb{D}$ and its compositional inverse is given by: 
\begin{equation*}
\alpha^{-1}(z) = \frac{4z}{(1+z)^2}.
\end{equation*}
Set
\begin{equation*}
y := H_0(z), \, z \in \mathbb{D}, \quad a(y) := \sqrt{\kappa^2 + (1-\kappa^2)y^2},
\end{equation*}
and recall from \cite{Biane} that 
\begin{equation*}
\xi_{2t}(u) := \frac{u-1}{u+1}e^{tu}
\end{equation*} 
is a analytic one-to-one map from a Jordan domain in the open right half-plane onto $\mathbb{D}$. Then, the flow $\psi_{\kappa,t}$ is given by
\begin{equation*}
\psi_{\kappa,t}(z) = \alpha\left(\frac{a^2(y)}{a^2(y)-\kappa^2}\alpha^{-1}[\xi_{2t}(a(y))]\right), \quad z \in (-1,z_{\kappa,t}),
\end{equation*}
and $z_{\kappa,t} \in (0,1)$ is the unique real solution of the equation: 
\begin{equation*}
\xi_{2t}[a(H_0(z))] = \frac{a(H_0(z)) - |\kappa|}{a(H_0(z)) + |\kappa|}.
\end{equation*}
Finally, $\psi_{\kappa,t}$ is locally invertible around $z=0$ with values in $\mathbb{D}$ and we denote below $\psi_{\kappa,t}^{-1}$ its local inverse. Using the variable change $z \mapsto a(H_0(z))$, the flow leads to the following map 
\begin{equation}\label{Func1}
\phi_{\kappa,t}: a \mapsto \alpha\left(\frac{a^2}{a^2-\kappa^2}\alpha^{-1}[\xi_{2t}(a)]\right)
\end{equation}
which is invertible near $a=1$. The Taylor coefficients of its local inverse were determined in \cite{Demni} and are given by: 
\begin{equation}\label{MomFlo}
c_n(\kappa,t) := \frac{2}{n}\sum_{k=1}^n(-1)^{n+k} \frac{2n}{n+k}\binom{n+k}{n-k} \sum_{j=1}^{k} \binom{2k}{k-j}e^{-jt} \sum_{m= 0}^{j-1}L_{j-m-1}^{(m+1)}(2jt)P_k^{(m)}(\kappa^2), \quad n \geq 1.
\end{equation}
Here, $L_k^{(\gamma)}$ is the $k$-th Laguerre polynomial of parameter $\gamma > -1$ (\cite{AAR}) and 
\begin{equation*}
P_k^{(m)}(\kappa^2) := \frac{(-2)^m}{m!}\sum_{j=0}^k\binom{k}{j}(-\epsilon)^j(2j)_m, \, \, \, (0)_m := \delta_{m0}, \,\,\, (2j)_m := \frac{\Gamma(2j+m)}{\Gamma(2j)}, j \geq 1.
\end{equation*}
In the sequel, we shall use the notations `$a$' as a variable and $a(\cdot)$ as the above function, and hope there is no confusion. 
 
\section{Inverse of the flow: Analytic extension, univalence and Aleksandrov-Clark measures} 
In this section, we prove that the flow $\psi_{\kappa,t}$ extends to a univalent function in the connected component of its analyticity region containing the origin onto $\mathbb{D}$. 
As mentioned in the introductory part, this result is a consequence of the general theory of L\"owner equations. 
\begin{pro}
Let $t > 0$ and $\kappa \in (-1,1)$. Let
\begin{align*}
\Sigma_{\kappa, t} &:= \{\psi_{\kappa,t}(z) \in \mathbb{D}\}
\\& = \left\{z \in \mathbb{D}, \, \frac{a^2(y)}{a^2(y)-\kappa^2}\alpha^{-1}[\xi_{2t}(a(y))] \in \mathbb{C} \setminus [1,\infty[, \,\, y = H_0(z)\right\} 
\end{align*}
be the analyticity region of $\psi_{\kappa, t}$ and denote $\Lambda_{\kappa, t}$ its connected component containing the origin $z=0$. Then $\psi_{\kappa,t}$ is a one-to-one map from $\Lambda_{\kappa, t}$ onto $\mathbb{D}$. 
\end{pro}
\begin{proof} 
The flow $\psi_{\kappa,t}$ solves the coupled ordinary differential equations (see the proof of Theorem 4.1 in \cite{Dem-Hmi}):
\begin{equation}\label{Low}
\partial_t\psi_{\kappa,t} = \psi_{\kappa,t} H_t(\psi_{\kappa,t}), \quad \psi_{\kappa,0}(z) = z,
\end{equation}
\begin{equation*}
\partial_t \left[H_{\kappa,t}(\psi_{\kappa,t})\right] = 2\kappa^2\frac{\psi_{\kappa,t}(1+\psi_{\kappa,t})}{(1-\psi_{\kappa,t})^3}. 
\end{equation*}
The differential equation \eqref{Low} is the radial L\"owner equation driven by the Herglotz function $H_{\kappa,t}$. As a matter of fact, if $T_{z}$ is the supremum of all $t$ such that $\psi_{\kappa,t}(z) \in \mathbb{D}$ for fixed $z \in \mathbb{D}$ (the lifespan of $t \mapsto \psi_{\kappa,t}(z)$), then $\psi_{\kappa,t}$ is a analytic one-to-one map from $\{z, T_{z} > t\}$ onto $\mathbb{D}$  (see Theorem 4.14 in \cite{Law} and its proof). Since $T_0 = \infty$ then $\{z, T_z > t\} = \Lambda_{\kappa, t}$ and the proposition is proved. 
\end{proof}

Now, we shall use the characteristics equation \eqref{Dyn} to relate the flow's inverse to $H_{\kappa,t}$ and $H_0$ in such a way that it defines the Aleksandrov-Clark measure at $z=1$ of a probability measure on the unit circle.  
\begin{pro}
The inverse $\psi_{\kappa,t}^{-1}$ of $\psi_{\kappa,t}$ satisfies 
\begin{equation}\label{Ext}
H_{0}(\psi_{\kappa,t}^{-1}(z)) = \frac{\sqrt{[H_{\kappa,t}(z)]^2 - \kappa^2[H_0(z)]^2}}{\sqrt{1-\kappa^2}}
\end{equation}
in the open unit disc. Moreover, there exists a unique probability measure $N_{\kappa,t}$ in $\mathbb{T}$ such that 
\begin{equation*}
H_{0}(\psi_{\kappa,t}^{-1})(z) = \int_{\mathbb{T}}\frac{w+z}{w-z}N_{\kappa,t}(dw).
\end{equation*}
\end{pro}

\begin{proof} 
Rewrite \eqref{Dyn} as 
\begin{equation*}
[H_{\kappa,t}(z)]^2 - \kappa^2[H_0(z)]^2 = (1-\kappa^2) [H_{0}(\psi_{\kappa,t}^{-1}(z))]^2,
\end{equation*}
near $z=0$ and note that this equality extends analytically to $\mathbb{D}$. Next, recall from \cite{Dem-Hmi} (see Corollary 4.5) that $\nu_{\kappa,t}\{1\} = |\kappa|$. It follows that 
\begin{equation*}
H_{\kappa,t}(z) - |\kappa|H_0(z), \quad z \in \mathbb{D},
\end{equation*}
is the Herglotz transform of the finite positive measure $\nu_{\kappa,t} - |\kappa|\delta_1$. Obviously, this claim holds true for 
\begin{equation*}
H_{\kappa,t}(z) + |\kappa|H_0(z)
\end{equation*}
and the finite positive measure $\nu_{\kappa,t} + |\kappa|\delta_1$. Since the Herglotz transform of a positive measure in $\mathbb{T}$ is analytic in $\mathbb{D}$ and have positive real part, then  
\begin{equation*}
[H_{\kappa,t}(z)]^2 - \kappa^2[H_0(z)]^2 = [H_{\kappa,t}(z) - |\kappa| H_0(z)][H_{\kappa,t}(z) + |\kappa| H_0(z)]
\end{equation*}
is analytic in $\mathbb{D}$ as well and can not take negative values there. Taking the principal branch of the square root yields \eqref{Ext} and its right-hand side is a Herglotz function. The existence of $N_{\kappa,t}$ follows then from Herglotz representation Theorem (\cite{CMR}, p.32). 
\end{proof}

\begin{rem}
Using \eqref{Pde}, we readily get:
\begin{align*}
\partial_tH_{0}(\psi_{\kappa,t}^{-1}) = \partial_t \frac{\sqrt{H_{\kappa,t}^2 - \kappa^2H_0^2}}{\sqrt{1-\kappa^2}} & = \frac{H_{\kappa,t} \partial_tH_{\kappa,t}}{\sqrt{1-\kappa^2}\sqrt{H_{\kappa,t}^2 - \kappa^2H_0^2}} 
= -zH_{\kappa,t} \partial_z H_{0}(\psi_{\kappa,t}^{-1}),
\end{align*}
which is also referred to as the radial L\"owner equation with driving function $H_{\kappa,t}$ (see the bottom of p.97-98 in \cite{Law}). On the other hand, the relation between L\"owner equations and free convolution semi-groups has been investigated in \cite{Bau}. According to Theorem 3 there, $N_{\kappa,t}$ is the image of $\nu_{\kappa,t}$ under the map induced by the above radial L\"owner equation.  
\end{rem}
%Now,
%\begin{equation*}
%\partial_z \left[H_{0}(\psi_{\kappa,t}^{-1})\right] (z) = 0\Leftrightarrow \partial_z \sqrt{[H_{\kappa,t}]^2 - \kappa^2[H_0]^2}(z) = 0 \Leftrightarrow \partial_z [H_{\kappa,t}]^2(z) =\kappa^2\partial_z [H_0]^2(z).
%\end{equation*}
%Moreover, the Herglotz transform $H_{\kappa,t}$ satisfies the pde (\cite{Dem-Hmi})
%\begin{equation*}
%\partial_t H_{\kappa,t}(z) +\frac{z}{2}\partial_z [H_{\kappa,t}(z)]^2 =\frac{z}{2}\kappa^2\partial_z [H_0(z)]^2.
%\end{equation*}
%Hence, if for fixed $t_0 > 0$ there exists $z = z_{t_0} \in \mathbb{D}$ such that $\partial_z\left[H_{0}(\psi_{\kappa,t}^{-1})\right](z_{t_0}) = 0$ then $\partial_t \left[H_{\kappa,t}\right](z_{t_0})=0$. Consequently, this complex number satisfies 
%\begin{equation*}
%H_{\infty}(z) = H_0(z) 
%\end{equation*}
%whence we readily deduce that $z=0$. Contradiction since $\partial_z\left[H_{0}(\psi_{\kappa,t}^{-1})\right](0) = 2\partial_z\psi_{\kappa,t}^{-1}(0) \neq 0$. The proposition is proved. 

%In the sequel, we shall keep the same notation  
\begin{rem}
The probability measure $N_{\kappa,t}$ is the so-called Alekasandrov-Clark measure associated with $\psi_{\kappa,t}^{-1}$ at $z=1$ (\cite{CMR}, p. 201-202). It is invariant under conjugation since
\begin{equation*}
\overline{H_{0}\left(\psi_{\kappa,t}^{-1}(z)\right)} = H_{0}(\psi_{\kappa,t}^{-1}(\overline{z}))
\end{equation*}
and satisfies $N_{\kappa,0} = \delta_1, N_{\kappa,\infty} = d\theta/(2\pi)$. More generally, the Aleksandrov-Clark probability measure $N_{\kappa,t,\zeta}$ corresponding to $\zeta \in \mathbb{T}$ is defined by its Herglotz transform as (see Prop. 9.1.6 in \cite{CMR}):
\begin{equation*}
\frac{\zeta + \psi_{\kappa,t}^{-1}(z)}{\zeta - \psi_{\kappa,t}^{-1}(z)} := \int_{\mathbb{T}}\frac{w+z}{w-z}N_{\kappa,t, \zeta}(dw).
\end{equation*}
\end{rem}

If $\kappa = 0$, then $N_{0,t} = \nu_{0,t}$ is the spectral distribution of $Y_{2t}$ which is known to have a continuous density (see \cite{Biane} for further details). Otherwise, we can prove a weaker result which follows from the following technical proposition. 
\begin{pro}\label{prop3}
Consider the set 
\begin{equation*}
\{\Re(a) \geq 0, \, |\phi_{\kappa,t}(a)| < 1\}
\end{equation*}
where $\phi_{\kappa,t}$ is given by \eqref{Func1} and denote $G_{\kappa,t}$ its connected component containing $a=1$. Then, the boundary of $G_{\kappa,t}$ is contained in a strip $\{0 \leq \Re(a) \leq A_{2t}\}$ for some $A_t > 1$ depending only on $t$ and bounded values of $|\Im(a)|$ near the imaginary axis . 
\end{pro}
\begin{proof}
Consider the map: 
\begin{equation*}
g_{\kappa,t}(a) := \frac{a^2}{a^2-\kappa^2}\alpha^{-1}[\xi_{2t}(a)], \quad \Re(a) \geq 0,
\end{equation*}
and for each $t$, denote $b_{2t} > 1$ the unique positive real satisfying
\begin{equation*}
\xi_{2t}(a) = \frac{a-1}{a+1}. 
\end{equation*}
It is clear that $b_{2t} = H_0(z_{0,t})$  and we already know from \cite{Biane} that $|\xi_{2t}(a)| > 1$ for $\Re(a) > b_{2t}$. Hence,
\begin{align*}
|g_{\kappa,t}(a)| &= \frac{|a|^2}{|a+1|^2} \frac{|a^2-1|}{|a^2-\kappa^2|} \frac{4e^{t\Re(a)}}{|1+\xi_{2t}(a)|^2}
\\& \leq \left(1+\frac{(1-\kappa^2)}{|a^2-\kappa^2|}\right) \frac{4e^{t\Re(a)}}{(|\xi_{2t}(a)|-1)^2}
\end{align*}
where we used the inequality $|1+\xi_{2t}(a)| \geq |1-|\xi_{2t}(a)|$ and the triangular inequality. Using further the inequalities 
\begin{eqnarray*}
|a^2-\kappa^2| & > (\Re(a) - \kappa^2) > (1-\kappa^2), \\
& |\xi_{2t}(a)| \geq |\xi_{2t}(\Re(a))| > 1,
\end{eqnarray*}
we get 
\begin{align*}
|g_{\kappa,t}(a)| \leq  \frac{8e^{t\Re(a)}}{(|\xi_{2t}(\Re(a))|-1)^2}.
\end{align*}
Since the right hand side of the last inequality tends to zero when $\Re(a) \rightarrow +\infty$, then $|g_{\kappa,t}(a)|$ does so uniformly in $\Im(a)$ and $\kappa$. Thus, there exists a positive real number $A_{2t} > b_{2t}$ independent from $\kappa$ and such that 
$|g_{\kappa,t}(a)| < 1, \Re(a) \geq A_{2t}$. Since 
\begin{equation*}
\partial G_{\kappa,t} = \{g_{\kappa,t}(a) \in [1,\infty[\} \subset \{|g_{\kappa,t}(a)| \geq 1\},
\end{equation*}
then the first statement is clear. 

Next, consider the equation  
\begin{equation}\label{Bound1}
\Im(g_{\kappa,t}(a)) = 0
\end{equation} 
and observe that $\partial G_{\kappa,t} \subset \{\Im(g_{\kappa,t}(a)) = 0\}$. Observe also that $\Im(g_{\kappa,t}(a)) =0 \Leftrightarrow  \Im(g_{\kappa,t}(\overline{a})) = 0$ so that we may restrict \eqref{Bound1} to $\{\Re(a) \geq 0, \Im(a) \geq 0\}$. Now, write 
\begin{align*}
g_{\kappa,t}(a) &= \frac{4a^2(a^2-1)e^{ta}}{(a^2-\kappa^2)(a+1+(a-1)e^{ta})^2} 
\\& = 4 \frac{a^2e^{ta} (|a|^4+\kappa^2 - \kappa^2a^2 - \overline{a}^2)((\overline{a}+1)^2 + (\overline{a}-1)^2e^{2t\overline{a}} + 2(\overline{a}^2-1)e^{t\overline{a}})}{\left|(a^2-\kappa^2)(a+1+(a-1)e^{ta})^2\right|^2}
\end{align*}
and substitute $a := x+iu, u \geq 0$. Since $x=0$ and $u=0$ are trivial solutions to \eqref{Bound1}, then this equation is equivalent to
\begin{align*}
 \frac{\Im[\Phi_{\kappa,2t}(x+iu)]}{ux}=0, 
\end{align*}
on $\{x \neq 0, u \neq 0\}$, which is expanded as
\begin{align}\label{EQ2}
0 = &-4 \kappa^2  \left((x^2+u^2+1)^2 - 4x^2\right) e^{tx}
\\&+\frac{\left(\kappa^2-(\kappa^2+1) (x^2-u^2)+(x^2+u^2)^2\right) }{xu}\left[2 x \left(e^{2 t x}+1\right) \left(\left(x^2+u^2\right) \sin (t u)+u \cos (t u)\right)\right. \nonumber
\\&\left.+\left(1-e^{2 t x}\right) \left(\left(\left(x^2+u^2\right)^2+x^2-u^2\right) \sin (t u)+2 u \left(x^2+u^2\right) \cos (t u)\right)\right]\nonumber
\\&+2 \left(1-\kappa^2\right) \left[\left(1-e^{2 t x}\right) \left(2 x \left(x^2+u^2\right) \cos (t u)-2 x u \sin (t u)\right)\right.\nonumber
\\&\left.+\left(e^{2 t x}+1\right) \left(\left(\left(x^2+u^2\right)^2+x^2-u^2\right) \cos (t u)-2 u \left(x^2+u^2\right) \sin (t u)\right)\right]. \nonumber
\end{align}

Taking the limit as $x\rightarrow0^+$ in \eqref{EQ2}, we get
\begin{align*}
0 = &-4 \kappa^2  (u^2+1)^2 
\\&+2(1+u^2)(\kappa^2+u^2) \left[u(t+2-tu^2) \sin (t u)+2 (1-tu^2) \cos (t u)\right]
\\&+4u^2 (1-\kappa^2) \left[(u^2-1) \cos (t u)-2 u \sin (t u)\right]
\end{align*}

\begin{align*}
= &-4 \kappa^2  (u^2+1)^2 
\\&+\left[2u(1+u^2)(\kappa^2+u^2) (t+2-tu^2)-8u^3 (1-\kappa^2) \right]\sin (t u)
\\&+\left[4u^2 (1-\kappa^2) (u^2-1)+4(1+u^2)(\kappa^2+u^2)  (1-tu^2)\right]\cos (t u).
\end{align*}
Using the basic trigonometric identities:
\begin{align*}
1=\cos \left(\frac{t u}{2}\right)^2+\sin \left(\frac{t u}{2}\right)^2 ,\quad \sin (t u)=2\sin \left(\frac{t u}{2}\right) \cos \left(\frac{t u}{2}\right), \quad \cos (t u)=\cos \left(\frac{t u}{2}\right)^2-\sin \left(\frac{t u}{2}\right)^2,
\end{align*}
we further derive
\begin{align*}
0 = &-4 \kappa^2  (u^2+1)^2 \left( \cos \left(\frac{t u}{2}\right)^2+\sin \left(\frac{t u}{2}\right)^2 \right)
\\&+\left[2u(1+u^2)(\kappa^2+u^2) (t+2-tu^2)-8u^3 (1-\kappa^2) \right]2\sin \left(\frac{t u}{2}\right) \cos \left(\frac{t u}{2}\right)
\\&+\left[4u^2 (1-\kappa^2) (u^2-1)+4(1+u^2)(\kappa^2+u^2)  (1-tu^2)\right]\left(\cos \left(\frac{t u}{2}\right)^2-\sin \left(\frac{t u}{2}\right)^2 \right),
\\ &= \left[4u^2 (1-\kappa^2) (u^2-1)+4(1+u^2)(\kappa^2+u^2)  (1-tu^2)-4 \kappa^2  (u^2+1)^2\right] \cos \left(\frac{t u}{2}\right)^2
\\&+2\left[2u(1+u^2)(\kappa^2+u^2) (t+2-tu^2)-8u^3 (1-\kappa^2) \right]\sin \left(\frac{t u}{2}\right) \cos \left(\frac{t u}{2}\right)
\\&-\left[4u^2 (1-\kappa^2) (u^2-1)+4(1+u^2)(\kappa^2+u^2)  (1-tu^2)+4 \kappa^2  (u^2+1)^2\right]\sin \left(\frac{t u}{2}\right)^2.
\end{align*}
Equivalently,
\begin{align*}
0 = &\left[ (1-\kappa^2) (u^2-1)+(1+u^2)  (1-tu^2-(t+1)\kappa^2)\right]u^2 \cos \left(\frac{t u}{2}\right)^2
\\&+\left[(1+u^2)(\kappa^2+u^2) (t+2-tu^2)-4u^2 (1-\kappa^2) \right]u\cos \left(\frac{t u}{2}\right) \sin \left(\frac{t u}{2}\right)
\\&-\left[u^2 (1-\kappa^2) (u^2-1)+(1+u^2)(u^2 (1-tu^2)+ \kappa^2  (2+u^2-tu^2))\right]\sin \left(\frac{t u}{2}\right)^2
\\ & = -\left[ \kappa^2 \left(t u^2+t+2 u^2\right)+u^2 \left(t u^2+t-2\right)\right]u^2 \cos \left(\frac{t u}{2}\right)^2
\\&-\left[\kappa^2 (tu^4-6u^2-t-2)+u^2 (tu^4-2u^2-t+2) \right]u\cos \left(\frac{t u}{2}\right) \sin \left(\frac{t u}{2}\right)
\\&+\left[\kappa^2 \left(t u^4+(t-4) u^2-2\right)+u^4 \left(t u^2+t-2\right)\right]\sin \left(\frac{t u}{2}\right)^2.
\\ &= - \left(u \cos \left(\frac{tu}{2}\right)-\sin \left(\frac{t u}{2}\right)\right) \left[u \left(\kappa^2 \left(t u^2+t+2 u^2\right)+u^2 \left(t u^2+t-2\right)\right) \cos \left(\frac{t u}{2}\right)\right.
\\&\left.+\left(\kappa^2 \left(t u^4+(t-4) u^2-2\right)+u^4 \left(t u^2+t-2\right)\right) \sin \left(\frac{t u}{2}\right)\right].
\end{align*}
The first factor of the last expression vanishes on an unbounded denumerable (hence totally disconnected) set while the second one is unbounded for large $u > 0$. Since $G_{\kappa,t}$ is connected, the second statement of the proposition follows and the proof is complete. 
%\begin{align*}
%u \left(\kappa^2 \left(t u^2+t+2 u^2\right)+u^2 \left(t u^2+t-2\right)\right) \cos \left(\frac{t u}{2}\right)+\left(\kappa^2 \left(t u^4+(t-4) u^2-2\right)+u^4 \left(t u^2+t-2\right)\right) \sin \left(\frac{t u}{2}\right)=0.
%\end{align*}
%Equivalently,
%\begin{align*}
%\kappa^2= \frac{-u^3 \left(t u^2+t-2\right) \left(u \sin \left(\frac{t u}{2}\right)+\cos \left(\frac{t u}{2}\right)\right)}{u \left((t+2) u^2+t\right) \cos \left(\frac{t u}{2}\right)+\left(t u^4+(t-4) u^2-2\right) \sin \left(\frac{t u}{2}\right)}.
%\end{align*}
%Since $\kappa^2\in(0,1)$, then the following inequality
%\begin{align*}
%1\geq \frac{u^3\left| \left(t u^2+t-2\right) \left(u \sin \left(\frac{t u}{2}\right)+\cos \left(\frac{t u}{2}\right)\right)\right|}{|u \left((t+2) u^2+t\right)|+\left|t u^4+(t-4) u^2-2\right|}
%\longrightarrow_{y\rightarrow\infty}\infty
%\\&=\frac{y^3\left| \left(t y^2+t-2\right) \left(y \sin \left(\frac{t y}{2}\right)+\cos \left(\frac{t y}{2}\right)\right)\right|}{-t y^4+(t+2) y^3-(t-4) y^2+ty+2}.
%\end{align*}
%holds for $u> 0$ and forces $u$ to take bounded values. A similar claim holds for $u < 0$ which completes the proof of the lemma. 
\end{proof}

\begin{cor}
Let $H^{\infty}$ be the Hardy space of bounded functions in $\mathbb{D}$: 
\begin{equation*}
H^{\infty} := \{f: \mathbb{D} \rightarrow \mathbb{C}, \, \sup_{z \in \mathbb{D}} |f(z)| < \infty\}.
\end{equation*}
Then, $N_{\kappa,t}, t > 0, \kappa \neq 0$ is absolutely continuous with respect to Lebesgue measure on $\mathbb{T}$ and its density belongs to $L^{\infty}(\mathbb{T})$. 
\end{cor}
\begin{proof}
With regard to \eqref{Ext} and according to Theorem 1.7, p.208, in \cite{Con}, this corollary is equivalent to $H_{0}(\psi_{\kappa,t}^{-1}) \in H^{\infty}$ which is in turn equivalent to the fact that the boundary $\Lambda_{\kappa, t}$ is at a positive distance from $z=1$. To prove the last claim, we reformulate the previous proposition in the $z$-variable since
\begin{equation*}
\psi_{\kappa,t}(z) = \phi_{\kappa,t}(a(H_0(z)), \quad z \in \overline{\mathbb{D}},
\end{equation*}
which we can achieve using some elementary geometrical considerations. More precisely, the inverse image of the strip $\{\Re(a) > A_{2t}\}$ under the map $z \mapsto a(H_0(z))$ is a region in $\mathbb{D} \cup \{1\}$ which lies inside a symmetric (with respect to the real axis) cone with vertex at $z=1$. Moreover, since the inverse image of the imaginary axis under this map is $\mathbb{T} \setminus \{1\}$, then the second statement of the previous lemma shows that $\overline{\Lambda_{\kappa, t}}$ does not contain two symmetric and disjoint arcs in $\mathbb{T}$ separated only by $z=1$. By continuity, $\overline{\Lambda_{\kappa, t}}$ does not contain two symmetric and disjoint bands near these arcs and intersecting the aforementioned cone (and not the segment $(-1,1)$, see the figure below). In summary, the boundary of $\Lambda_{\kappa, t}$ can not approach $z=1$ in $\overline{\mathbb{D}}$.   
\begin{figure}[!h]
\centering
\includegraphics[width=0.2\textwidth]{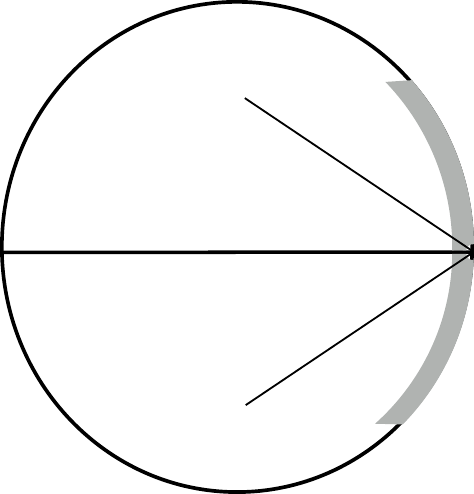}
\caption{}
\end{figure}
\end{proof}

We close this section by writing a direct proof of the univalence of $\psi_{\kappa,t}^{-1}$ in $\mathbb{D}$ without appealing to the theory of L\"owner equations. To proceed, note first that the right-hand side of \eqref{Ext} can be taken as an analytic extension of 
$\psi_{\kappa,t}^{-1}$ from a neighborhood of $z=0$ to $\mathbb{D}$. In particular, the equality 
\begin{equation*}
\psi_{\kappa, t}^{-1}(\psi_{\kappa,t}(z)) = z 
\end{equation*}
holds on $\Lambda_{\kappa, t}$ whence we infer that $\psi_{\kappa,t}$ is injective in $\Lambda_{\kappa,t} \subset \psi_{\kappa,t}^{-1}(\mathbb{D})$. Hence, it remains to prove that $\psi_{\kappa, t}(\Lambda_{\kappa,t}) = \mathbb{D}$. To this end, set
\begin{equation}
 K_{\kappa,2t}(z) :=\sqrt{[H_{\kappa,t}(z)]^2 + \kappa^2(1-[H_0(z)]^2)} = a[H_0(\psi_{\kappa,t}^{-1}(z))],
\end{equation}
which defines an analytic map in $\mathbb{D}$.

\begin{lem}\label{Lemme1}
Near the origin $z=0$, the equality $\psi_{\kappa,t}(\psi_{\kappa,t}^{-1}(z)) = z$ is equivalent to
\begin{equation}\label{A1}
\xi_{2t}(K_{\kappa,2t}(z))=\frac{H_0(z)K_{\kappa,2t}(z)-H_{\kappa,t}(z)}{H_0(z)K_{\kappa,2t}(z)+H_{\kappa,t}(z)}.
\end{equation}
\end{lem}
\begin{proof}
Recall
\begin{equation}\label{Surj}
\psi_{\kappa,t}(\psi_{\kappa,t}^{-1}(z)) = \alpha\left[\frac{[K_{\kappa,2t}(z)]^2}{[K_{\kappa,2t}(z)]^2-\kappa^2} \alpha^{-1}[\xi_{2t}(K_{\kappa,2t}(z))]\right]=z
\end{equation}
in a neighborhood of $z=0$. Using
\begin{equation*}
\alpha^{-1}(z) = 1 - \frac{1}{H_0^2(z)},
\end{equation*}
then \eqref{Surj} may be written as 
\begin{equation*}
1 - \frac{1}{H_0^2(\xi_{2t}(K_{\kappa,2t}(z)))}  = \frac{([K_{\kappa,2t}(z)]^2 - \kappa^2)([H_0(z)^2]-1)}{[K_{\kappa,2t}(z)]^2[H_0(z)]^2}
\end{equation*}
or equivalently as
\begin{equation*}
H_0^2(\xi_{2t}(K_{\kappa,2t}(z))) =  \frac{[K_{\kappa,2t}(z)]^2[H_0(z)]^2}{[H_{\kappa,t}(z)]^2}.
\end{equation*}
Since $H_0(0) = H_{\kappa,t}(0) = 1$ then $\xi_{2t}(K_{\kappa,2t}(0)) = \xi_{2t}(1) = 0$ so that
\begin{equation*}
H_0(\xi_{2t}(K_{\kappa,2t}(z))) =  \frac{K_{\kappa,2t}(z)H_0(z)}{H_{\kappa,t}(z)}
\end{equation*}
in a neighborhood of $z=0$. The last equality may be written as 
\begin{equation*}
\xi_{2t}(K_{\kappa,2t}(z))=\frac{H_0(z)K_{\kappa,2t}(z)-H_{\kappa,t}(z)}{H_0(z)K_{\kappa,2t}(z)+H_{\kappa,t}(z)}
\end{equation*}
which proves the lemma.
\end{proof}
Though the equivalence stated in lemma \ref{Lemme1} holds in a neighborhood of $z=0$, what we really need to prove that $\psi_{\kappa, t}(\Lambda_{\kappa,t}) = \mathbb{D}$ is that the second equality implies the first one. More precisely, assume there exists $z \in \mathbb{D} \setminus \{0\}$ such that $H_0(z)K_{\kappa,2t}(z)+H_{\kappa,t}(z) = 0$. Then, straightforward computations show that this complex number satisfies
\begin{equation*}
[H_{\kappa,t} - |\kappa|H_0(z)][H_{\kappa,t} + |\kappa|H_0(z)](1-H_0^2(z)) = 0
\end{equation*}
which can not hold since $H_0(z) \neq 1$ and since
\begin{equation*}
\Re[H_{\kappa,t} \pm |\kappa|H_0(z)] > 0. 
\end{equation*}
But $H_0(0)K_{\kappa,2t}(0)+H_{\kappa,t}(0) = 2$ therefore both sides of \eqref{A1} are analytic in the open unit disc $\mathbb{D}$. Reading backward the proof of lemma \ref{Lemme1}, we conclude that 
\begin{equation*}
\frac{[K_{\kappa,2t}(z)]^2}{[K_{\kappa,2t}(z)]^2-\kappa^2} \alpha^{-1}[\xi_{2t}(K_{\kappa,2t}(z))] = \alpha^{-1}(z)
\end{equation*}
for any $z \in \mathbb{D}$, in particular, the range of the left-hand side is exactly $\mathbb{C} \setminus [1,\infty[$. Composing both sides of the last equality with the map $\alpha$ yields $\psi_{\kappa, t}(\Lambda_{\kappa,t}) = \mathbb{D}$.

\section{Consequences}
Here, we derive two corollaries of our preceding results on the spectra of $U_t$ and $J_t$. The first of them is: 
\begin{cor}
For any time $t >0$ and any $\kappa \neq 0$, $z=1$ does not belong neither to the absolutely-continuous spectrum of $\nu_{\kappa,t}$ nor to its continuous singular part. 
\end{cor}
\begin{proof}
Recall from \cite{Dem-Hmi}, Theorem 4.1, that $z_{\kappa,t} \in (0,1)$ is the right real boundary of $\Lambda_{\kappa,t}$. In particular, $\psi_{\kappa,t}(z_{\kappa,t}) = 1 \Leftrightarrow \psi_{\kappa,t}^{-1}(1) = z_{\kappa,t}$ so that \eqref{Ext} yields:
\begin{equation*}
H_0(z_t) = \lim_{z \rightarrow 1-}[H_{\kappa,t}(z) - |\kappa| H_0(z)][H_{\kappa,t}(z) + |\kappa| H_0(z)],
\end{equation*}  
where $z \rightarrow 1-$ is a radial limit. But $0 < H_0(z_{\kappa,t}) < \infty$ while $H_{\kappa,t}(z) + |\kappa| H_0(z)$ blows up as $z \rightarrow 1-$, therefore 
\begin{equation*}
\lim_{z \rightarrow 1-}[H_{\kappa,t}(z) - |\kappa| H_0(z)] = 0. 
\end{equation*}
Thus, the real part of $H_{\kappa,t}(z) - |\kappa| H_0(z)$, which is nothing else but the Poisson transform of $\nu_{\kappa,t} - |\kappa|\delta_1$, does so as well. But, the  atoms of this finite measure lie in $\mathbb{T} \setminus \{1\}$ and have a null contribution in the radial limit of the Poisson transform. Consequently, both the Poisson transforms of the density and of the continuous singular component of $\nu_{\kappa,t}$ vanishes as $z \rightarrow 1-$. Using Proposition 1.3.11 and equation (1.8.8) in \cite{CMR}, the corollary is proved.
\end{proof}

The second corollary is an explicit, yet involved, expression of the moments of $\nu_{\kappa,t}$: 
\begin{equation*}
\tau(U_t^n) = \int_{\mathbb{T}}w^n \nu_{\kappa,t}(dw), \quad n \geq 1,
\end{equation*}
which in turn leads to $\tau(J_t^n)/\tau(P), n \geq 1,$ by the virtue of \eqref{Binom}. More precisely,
\begin{cor}
The moments of $U_t$ are given by 
\begin{align*}
\tau(U_t^n) & =  1 - \int_0^t \left[nb_n(\kappa,t) + 2\sum_{k=1}^{n-1}kb_k(\kappa,s)b_{n-k}(\kappa,s)\right] ds
\\& = 1 - n\int_0^t \sum_{k=1}^{n}b_k(\kappa,s)b_{n-k}(\kappa,s) ds,
\end{align*}
where 
\begin{equation*}
b_n(\kappa,t) = \frac{c_n(\kappa,t)}{2}, \,n \geq 1, \quad  b_0(\kappa,t) \equiv 1,
\end{equation*}
$c_{n(\kappa,t}, n \geq 1,$ are displayed in \eqref{MomFlo}, and the sum in the RHS of the first equality is empty when $n = 1$. Consequently, 
\begin{equation*}
\frac{\tau(J_t^n)}{\tau(P)} = 1- \frac{2}{2^{2n}(1+\tau)} \int_0^t\sum_{k=1}^nk\binom{2n}{n-k} \sum_{j=1}^{k} b_j(\kappa,s)b_{k-j}(\kappa,s) ds. 
\end{equation*}
\end{cor}
\begin{proof}
%Since $\psi_{\kappa,t}(z) = \phi_{\kappa,t}[a(H_0(z))]$ then $a[H_0(\psi_{\kappa,t}^{-1})] = \phi_{\kappa,t}^{-1}$ at least locally around $z=0$. Thus, \eqref{Ext} yields
%\begin{equation*}
%\phi_{\kappa,t}^{-1}(z) = \sqrt{\kappa^2 + [H_{\kappa,t}(z)]^2 - \kappa^2[H_0(z)]^2}.
%\end{equation*}
Using the pde \eqref{Pde} and the fact that
\begin{equation*}
a(H_0(\psi_{\kappa,t}(z))) =\sqrt{[H_{\kappa,t}(z)]^2 + \kappa^2(1-[H_0(z)]^2)} = \phi_{\kappa,t}^{-1}(z),
\end{equation*}
is the inverse of $\phi_{\kappa,t}$, we deduce that 
\begin{equation}\label{EQ}
\partial_t  H_{\kappa,t}(z) = -\frac{z}{2} \partial_z[\phi_{\kappa,t}^{-1}]^2= -z\phi_{\kappa,t}^{-1}(z)\partial_z [\phi_{\kappa,t}^{-1}](z).
\end{equation}
Besides, \eqref{MomFlo} gives the Taylor expansion of $\phi_{\kappa,t}^{-1}$ around $z=0$ while \eqref{Herg} gives the moment expansion  of $H_{\kappa,t}$. Hence, the first expression of $\tau(U_t^n)$ follows after equating the Taylor coefficients of both sides of \eqref{EQ} and taking into account the initial values $c_0(\kappa,t) = 1, \tau(U_0) = 1$. In order to obtain the second one, it suffices to perform the index change $k \mapsto n-k$: 
\begin{equation*}
2\sum_{k=1}^{n-1}kb_k(\kappa,s)b_{n-k}(\kappa,s) = \sum_{k=1}^{n-1}kb_k(\kappa,s)b_{n-k}(\kappa,s) + \sum_{k=1}^{n-1}(n-k)b_k(\kappa,s)b_{n-k}(\kappa,s).
\end{equation*}
Finally, the expression of the moments of $J_t$ in $(P\mathcal{A}P, \tau/\tau(P))$ follows from the substitution  of $\tau(U_t^k), 1 \leq k \leq n$, in \eqref{Dyn} together with the binomial identity 
\begin{equation*} 
\binom{2n}{n} + 2\sum_{k=1}^n \binom{2n}{n-k} = 2^{2n}
\end{equation*}
and the relation $\tau(P) = (1+\kappa)/2$. 
\end{proof}

\begin{rem}
If $\kappa = 0$, then $P_k^{(m)}(0) = \delta_{m0}$ since $(0)_m = \delta_{m0}$ and as such  
\begin{equation*}
b_n(0,t) = \frac{e^{-nt}}{n}L_{n-1}^{(1)}(2nt), \quad n \geq 1. 
\end{equation*}
Moreover, it is already known that (see e.g. \cite{Rai}, p.669)
\begin{equation*}
\frac{d}{dt}b_n(0,t) = -nb_n(0,t) - n\sum_{k=1}^{n-1}b_k(\kappa,t)b_{n-k}(\kappa,t)
\end{equation*}
with $b_n(0,0) = 1$. As a result, $\tau(U_t^n) = b_n(0,t)$ which is in agreement with the fact that $U_t = SY_tSY_t^{\star}$ has the same spectral distribution as $Y_{2t}$ when $\tau(S) = 0$. 
\end{rem}

\section{Boundary behavior of $\psi_{\kappa,t}^{-1}$}
It is known that the boundary behavior of a univalent map $f$ in $\mathbb{D}$ depends on the regularity of the boundary of $f(\mathbb{D})$. For instance, the extension of $f$ to $\mathbb{T}$ is continuous if and only if the boundary of $f(\mathbb{D})$ is locally 
connected and it is further a homeomorphism from $\mathbb{T}$ onto $\partial f(\mathbb{D})$ if and only if $f(\mathbb{D})$ is a Jordan domain (see e.g. \cite{Pom}). Unless $\kappa=0$ in which case $\psi_{\kappa,t}$ reduces to $\xi_{2t}$, the relation \eqref{Ext} shows that if $\psi_{\kappa,t}$ extends continuously to $\mathbb{T}$ then $H_{\kappa,t} - |\kappa|H_0$ does so and as such, one gets the Lebesgue decomposition of $\nu_{\kappa,t}$. However, the main obstruction toward proving this regularity is that the boundary equation 
\begin{equation*}
|\psi_{\kappa,t}(z)| = 1
\end{equation*}
becomes highly nonlinear as seen for instance from \eqref{EQ2}. Even for the real left and right boundaries $\psi_{\kappa,t}(z) = \pm 1$, there can be more than one solution to the boundary equation. For instance, it is easily seen that
\begin{equation*}
\lim_{z \rightarrow -1, z > -1} \psi_{\kappa,t}(z) = -1. 
\end{equation*}
and this limit is also attained on the unit circle (tangentially) since $\xi_{2t}(0) = -1$ and since 
\begin{equation*}
a(y) = \sqrt{\kappa^2 + (1-\kappa^2)y^2}  
\end{equation*}
vanishes for some $y=H_0(z)$ lying in the imaginary axis. When $t > 2$, another real number lying in $(-1,0)$ yields this limit since there exists $d_{2t} \in (\sqrt{(t-2)/t},1)$ such that $\xi_{2t}(d_{2t}) = -1$. As to the real right boundary, recall that $z_{\kappa,t} \in (0,1)$ is the unique solution of:
\begin{equation*}
\xi_{2t}[a(H_0(z))] = \frac{a(H_0(z)) - |\kappa|}{a(H_0(z)) + |\kappa|}.
\end{equation*}
Since the map $\alpha^{-1}$ is invariant under inversion $u \mapsto 1/u$, then the unique real solution of 
\begin{equation*}
\xi_{2t}[a(H_0(z))] = \frac{a(H_0(z)) + |\kappa|}{a(H_0(z)) - |\kappa|}
\end{equation*}
satisfies $\psi_{\kappa,t}(z) = 1$ as well, yet it is larger than $z_{\kappa,t}$. As a matter of fact, the univalence of $\psi_{\kappa,t}$ in $\Lambda_{\kappa,t}$ forces to distinguish the cases $|\kappa| > d_{2t}, |\kappa| < d_{2t}$. In this respect, the intersection of $\Lambda_{\kappa,t}$ with $(-1,1)$ is described in the following lemma. 

\begin{lem}\label{Lemme2}
Recall the function
\begin{equation*}
g_{\kappa,t}(a) = \alpha^{-1} \circ \phi_{\kappa,t}(a) = \frac{a^2}{a^2-\kappa^2} \alpha^{-1}\left[(\xi_{2t}(a))\right]
\end{equation*}
of the real variable $a \in \left(|\kappa|, \infty\right)$. If $t \in (0,2]$, then $g_{\kappa,t}$ is an increasing from $\left(|\kappa|, a[H_0(z_{\kappa,t})]\right)$ onto $(-\infty,1)$. Otherwise, if $t > 2$, then it is so from $\left(|\kappa| \vee d_{2t}, a[H_0(z_{\kappa,t})]\right)$. 
\end{lem}
\begin{proof}
Since $\alpha$ and $z \mapsto a(H_0(z))$ are smooth and increasing maps in $ ]-\infty, 1[$ and in $(-1,1)$ respectively, then the behavior of 
\begin{equation*}
\psi_{\kappa,t} = \alpha \circ g_{\kappa,t} \circ a \circ H_0
\end{equation*}
in $(-1,1)$ is identical to that of $g_{\kappa,t}$ in $\left(|\kappa|, \infty\right)$. Then, let $t \in (0,2)$ and compute
\begin{equation}\label{Deriv}
(\xi_{2t})'(a) = \frac{e^{ta}}{(a+1)^2}[ta^2+2-t]
\end{equation}
to see that that $\xi_{2t}$ is increasing and maps $(0,  a[H_0(z_{0,t})])$ onto $(-1,1)$. Otherwise, $\xi_{2t}$ is increasing in the interval $\left(d_{2t}, a[H_0(z_{0,t}])\right)$ onto $(-1,1)$. Now, straightforward computations yield 
\begin{align*}
(g_{\kappa, t})'(a) &= \frac{a}{(a^2-\kappa^2)^2(1+\xi_{2t}(a))}\left[a(a^2-\kappa^2)(1-\xi_{2t}(a))(\xi_{2t})'(a) -2\kappa^2\xi_{2t}(a)(1+\xi_{2t}(a))\right]
\\& = \frac{ae^{ta}}{(a^2-\kappa^2)^2(1+\xi_{2t}(a))(1+a)^2}\left[a(a^2-\kappa^2)(1-\xi_{2t}(a))(t(a^2-1)+2) + 2\kappa^2(1-a^2)(1+\xi_{2t}(a))\right].
\end{align*}
If $t \leq 2$ and $a \in (|\kappa|, 1)$ then $(g_{\kappa,t})'(a) > 0$ and the same holds if $t \geq 2$ and $a \in (|\kappa| \vee d_{2t}, 1)$. Otherwise, let $t \geq 0$ and $a \in (1, a(H_0(z_{\kappa,t})))$. Then $t(a^2-1) > 0$ whence
\begin{multline*}
a(a^2-\kappa^2)(1-\xi_{2t}(a))(t(a^2-1)+2) + 2\kappa^2(1-a^2)(1+\xi_{2t}(a))  > \\ 2\left[a(a^2-\kappa^2)(1-\xi_{2t}(a)) -\kappa^2(a^2-1)(1+\xi_{2t}(a))\right]
\\ > 2(a^2-1)\left[a(1-\xi_{2t}(a)) -|\kappa|(1+\xi_{2t}(a))\right]
\end{multline*}
where the second inequality follows from $\kappa^2 < |\kappa| < 1$. But 
\begin{equation*}
a(1-\xi_{2t}(a)) -|\kappa|(1+\xi_{2t}(a)) = 0 \quad \Leftrightarrow \quad \xi_{2t}(a) = \frac{a-|\kappa|}{a+|\kappa|} \quad \Leftrightarrow \quad a =  a(H_0(z_{\kappa,t})).
\end{equation*}
As a result, $(g_{\kappa,t})'(a) > 0$ as well and the lemma is proved. 
\end{proof}

%So far, we dispose of an analytic extension of $\psi_{\kappa,t}^{-1}$ to $\mathbb{D}$. However, in order to extract the information on the spectrum of $U_t$ (and of $J_t$), we need to investigate the boundary behavior of $\psi_{\kappa,t}^{-1}$ on the unit circle or equivalently the regularity of the boundary of $\Lambda_{\kappa,t}$ (\cite{Pom}). In this respect, notice that $\psi_{\kappa,t}$ cannot attain the whole circle while being continuous due to the choice of the square root determination in the definition of $\alpha$. Rather, it may reach the lower half of the unit circle $\mathbb{T}$ which is sufficient for our purposes since the spectral distribution $\nu_{\kappa,t}$ of $U_t$ is invariant under complex conjugation $z \mapsto \overline{z}$. In particular, we already know from \cite{Dem-Hmi} that for any $t > 0$, there exists a unique real $z_{\kappa,t} \in (0,1)$ such that $\psi_{\kappa,t}(z_{\kappa,t}) = 1$ while 
%Note also that if 
%\begin{equation*}
%\lim_{a \rightarrow a_t, a > a_t} \psi_{\kappa,t}\left(\sqrt{H_0^{-1}\left\{\frac{a^2 - \kappa^2}{1-\kappa^2}\right\}}\right) = -1
%\end{equation*}
%as well provided $a_t > |\kappa|$. 
More generally, the boundary equation reads
\begin{equation*}
\psi_{\kappa,t}(z) = e^{i\theta}, \quad \theta \in (-\pi,0).
\end{equation*}
Here, the restriction on $\theta$ is due to the principal determination of the square root in the definition of $\alpha$. Nonetheless, the extension of $\psi_{\kappa,t}^{-1}$ to the lower part of $\mathbb{T}$ suffices to describe the continuous spectrum of $\nu_{\kappa,t}$ since this measure is invariant under complex conjugation. Set $\lambda = \alpha^{-1}(e^{i\theta}) \in [1,\infty)$ and use the variable change
\begin{equation*}
z \mapsto a = a(H_0(z)) \in \{\Re(z) \geq 0\} \setminus ]0,|\kappa|[. 
\end{equation*}
Then, the boundary equation is reformulated as:
\begin{equation}\label{bound1}
\phi_{\kappa,t}(a) = e^{i\theta} \quad \Leftrightarrow \quad \left[\frac{1}{2}\left(\frac{a-|\kappa|}{a+|\kappa|} + \frac{a + |\kappa|}{a-|\kappa|}\right) + 1\right] = \lambda\left[\frac{1}{2}\left(\xi_{2t}(a) + \frac{1}{\xi_{2t}(a)}\right) + 1\right]. 
\end{equation}

This reformulation together with Lemma \ref{Lemme2} suggest that $\partial \Lambda_{\kappa,t} \subset \{|\xi_{2t}(a)| \leq 1\}$, a result that we could not prove and that would considerably improve proposition \ref{prop3} if it holds true. 
Observe also that \eqref{bound1} leads to the following geometrical problem: consider a semi-line $\Delta_{\beta}$ in the lower quadrant of the right half-plane emanating from the origin and making an angle $\beta \in (-\pi/2, 0]$ with the real positive semi-axis:
\begin{equation*}
\Delta_{\beta} := \{k(1+i\tan(\beta)), k \geq 0\}.
\end{equation*}
Denote $\mathcal{C}_{\beta}$ the inverse image of $\Delta_{\beta}$ under the map
\begin{equation*}
a \mapsto \frac{1}{2}\left(a + \frac{1}{a}\right) + 1
\end{equation*}
and valued in $\mathbb{D}$. This is a curve starting nontangentially at $z=-1$ and approaching the origin as $k \rightarrow \infty$. Denote $\mathcal{C}_t$ and $\mathcal{C}_{\kappa}$ the images of $\mathcal{C}_{\beta}$ under $\xi_{2t}^{-1} = H_{0,t}$ and $|\kappa|H_0(z)$ 
(The inverse of $(z-|\kappa|)/(z+|\kappa|)$) respectively. Then, given $\lambda \geq 1$, is there a value of $\beta \in (-\pi/2, 0]$ such that $\mathcal{C}_t$ and $\mathcal{C}_{\kappa}$ have non empty intersection and at least one common point lying there satisfies \eqref{bound1}? Of course, if $\lambda = 1$ then $\beta = 0$. Otherwise, straightforward computations show that the curve $C_{\kappa}$ admits the following parametrization: 
\begin{equation*}
C_{\kappa} = \left\{|\kappa| \sqrt{\frac{k(1+i\tan \beta)}{k(1+i\tan \beta)-2}}, \quad k \geq 0\right\},
\end{equation*}
whence we see that $C_{\kappa}$ starts at $z=0$, comes close to $z=|\kappa|$ as $k \rightarrow \infty$ and 
\begin{equation*}
\max_{k \geq 0} \{|z|, z \in C_{\kappa}\} = \frac{|\kappa|}{\sqrt{|\sin \beta|}},
\end{equation*}
which is attained at $k=2$. As to $C_t, t < 2$, it starts at $z=0$ and comes close to $z=1$ as $k \rightarrow \infty$. Moreover, it is simple curve contained in the Jordan domain $\{|\xi_{2t}(a)| < 1\}$. Using \eqref{Deriv} together with the inverse relation $\xi_{2t}(H_{0,t}(z)) = z, z \in \mathbb{D}$, then we readily get\footnote{We consider the nontangential limit $z \rightarrow 1$.} 
\begin{equation*}
|\kappa|\partial_z[H_0](-1) = \frac{|\kappa|}{2} < \frac{1}{2-t} = \partial_z[H_{0,t}](-1), \quad t < 2, 
\end{equation*}
therefore $\mathcal{C}_{\kappa}$ lies under $\mathcal{C}_t$ in a neighborhood of $z=0$. As a result, there is an interval $]0,\beta_0[$ on which these curves has at least two common points. When $t \geq 2$, the curve $\mathcal{C}_t$ will start at the positive real $d_{2t}$ satisfying $\xi_{2t}(d_{2t}) = -1$ and similar computations show that 
\begin{equation*}
\partial_z[H_{0,t}](-1) = \frac{1}{\partial_z \xi_{2t}(d_{2t})} = \frac{1-(d_{2t})^2}{t(d_{2t})^2 + 2-t}.
\end{equation*}
In this case, we need a more careful study of the relative position of $\mathcal{C}_t$ with respect to $\mathcal{C}_{\kappa}$ (note that the function $t \mapsto d_{2t}$ is increasing and that $d_{\infty} = 1$). 

%Indeed, if $w$ satisfies $|\xi_{2t}(w)| > 1$ and solves \eqref{bound1} for some $\lambda \geq 1$, then $v:= H_{0,t}(1/\xi_{2t}(w))$ satisfies $|\xi_{2t}(v) | < 1$ since $\xi_{2t} = H_{0,t}^{-1}$ in $\{|\xi_{2t}(a)| \leq 1\}$ and 
%\begin{equation*}
%\xi_{2t}(v) + \frac{1}{\xi_{2t}(v)} = \xi_{2t}(w) + \frac{1}{\xi_{2t}(w)}. 
%\end{equation*}
%But, if $v$ solves \eqref{bound1} then we would have
%\begin{equation*}
%\frac{v^2}{v^2-\kappa^2} = \frac{w^2}{w^2-\kappa^2} \Leftrightarrow v = w
%\end{equation*}

\end{document}